\documentclass[12pt,a4paper]{amsart}
\usepackage{amssymb}

\usepackage[backref=true]{biblatex}

\usepackage{color}
\usepackage[usenames,dvipsnames,svgnames,table]{xcolor}
\usepackage[colorlinks=true,
            linkcolor=red,
            citecolor=blue]{hyperref}

\usepackage{graphicx,amssymb,amsfonts,epsfig,amsthm,a4,amsmath,url}
\usepackage[latin1]{inputenc}

\usepackage{amsmath, amsthm, amssymb, verbatim}
\usepackage{graphicx,amssymb,amsfonts,epsfig,amsthm,a4,amsmath,url,graphicx,enumerate}
\usepackage[latin1]{inputenc}
\usepackage{mathtools}
\usepackage{tikz}
\usepackage{textcomp}
\usetikzlibrary{matrix}

\usepackage{fancyhdr}
\fancypagestyle{plain}{%
	\fancyhead[R]{\fbox{to appear in journal xx}}
	
}
\usepackage{textpos}

\newtheorem{thm}{Theorem}[section]
\newtheorem{cor}[thm]{Corollary}
\newtheorem{lem}[thm]{Lemma}

\newtheorem{prop}[thm]{Proposition}
\theoremstyle{definition}
\newtheorem{defn}[thm]{Definition}

\theoremstyle{remark}
\newtheorem{rem}[thm]{Remark}

\numberwithin{equation}{section}

\newcommand{\Ind}{\text{Ind}}

\newcommand\restr[2]{{
  \left.\kern-\nulldelimiterspace 
  #1 
  \littletaller 
  \right|_{#2} 
  }}

\newcommand{\littletaller}{\mathchoice{\vphantom{\big|}}{}{}{}}

\begin{document}

\title[First $l^p$-cohomology and $L^q$-Measure Equivalence]{Invariance of non-vanishing of first $l^p$-cohomology under $L^q$-Measure Equivalence}

\author[Das]{Kajal Das}
\address{Department of Mathematics \\SRM University AP\\ Andhra Pradesh, India 522240}
\email{kdas.math@gmail.com, kajaldas.m@srmap.edu.in }

\maketitle
\textbf{Abstract:}  The first $l^p$-cohomology is an algebro-analytical object attached to a finitely generated discrete group and  introduced by M. Gromov. It is well known that it is invariant under quasi-isometry. In this article, we prove that the non-vanishing of the first $l^p$-cohomology of a non-amenable group is invariant under $L^q$-Measured Equiavalence , where $q\geq p>1 $.  We prove a weaker version of this result for $q\geq p=1$. We also discuss many corollaries of this result. Some corollaries are new and we reprove the other corollaries . We prove a new result that for hyperbolic (in the sense of Gromov)  groups with boundaries having Combinatorial Loewner Property, conformal dimension (of the canonical conformal gauge) of the Gromov boundary is invariant under $L^q$-Measure Equivalence  for some large $q$. We reprove that the finitely generated free groups and surface groups are not $L^1$-Measure Equivalent. Finally, we discuss $L^q$-Measure Equivalence between non-amenable 3-manifold groups corresponding to Thurston's three geometries 
$\mathbb{H}^3$, $\mathbb{H}^2\times\mathbb{R}$ and $\widetilde{SL_2(\mathbb{R})}$.

\vspace{1cm}


\vspace{1cm}

\textbf{Key terms:}: First $l^p$-cohomology of a group, $L^q$-Measure Equivalence, critical exponent for first $l^p$-cohomology of a group, conformal dimension,  hyperbolic  groups with boundaries having Combnatorial Lowener Property. 

\tableofcontents

\section{Introduction}

\textit{Measure equivalence} (ME) is an equivalence relation on countable groups introduced by Gromov in \cite{Gro93}, as a measure-theoretic analogue of  quasi-isometry (QI). The first detailled study of ME was performed in the work of Furman \cite{Fur99a} in the context of ME-rigidity of lattices in higher rank  simple Lie groups. \textit{$L^p$-measure equivalence} ($L^p$-ME) is defined by imposing an $L^p$-condition on the \textit{cocycle maps} arising from a measure equivalence relation. Such integrability condition first appeared in Margulis's proof of the normal subgroup theorem for irreducible lattices \cite{Ma79}.  $L^p$-measure equivalence is an emerging area in the overlap geometric group theory and measured group theory. The main question in $L^p$-measured group theory is of the following type : 
\begin{center}
  \textit{Which algebraic/geometric properties of groups are invariant under $L^p$-ME?}
\end{center}

In recent history,  there are major works on $L^p$-measure equivalence,  by Y. Shalom in \cite{Sh00} , \cite{Sh00'} (for lattices), \cite{Sh04} (in terms of $L^\infty$-Measure Equivalence)  and later by Bader-Furman-Sauer \cite{BFS13}( in general $L^p$-Measure Equivalence between discrete groups).  Subsequently, it appeared in the work of  T. Austin ( L. Bowen, appendix) \cite{Aus16} , Das-Tessera \cite{DT18}, 
Das \cite{Das18}, \cite{Das18'}, T. Austin \cite{Aus16'}, L. Bowen \cite{Bow17}, Delabie-Koivisto-Ma\^{i}tre-Tessera \cite{DKMT22} and Horbez-Huang\cite{HH23}.  We briefly survey the recent developments . 

\begin{itemize}

\item Shalom has used $L^2$-measure equivalence in \cite{Sh00'} to induce cocycle from unitary representations of a lattice  to the ambient locally compact second countable group. In \cite{Sh04} he proves that if two amenable groups are quasi-isometric, then they are  $L^\infty$-measure equivalence. 

\item In \cite{BFS13} Bader- Furman-Sauer prove  $L^1$-measure equivalence rigidity for lattices in $SO(n,1)$, $n\geq 3$. 

\item  T. Austin proves that  if two nilpotent groups are $L^1$-ME, then their asymptotic cones are bi-Lipschitz isomorphic \cite{Aus16}. In the appendix,  L. Bowen proves that if two finitely generated groups are $L^1$-ME, then they have same `growth types'. 

\item In \cite{DT18} the authors prove that quasi-isometry and Measure Equivalence together  do not imply $L^1$ or $L^\infty$-measure equivalence.

\item L. Bowen proves $L^1$-orbit equivalence rigidity (a variant of Measure Equivalence) for free groups inside the class of accessible' groups \cite{Bow17}. 

\item  T. Austin proves the  invariance of entropy (up to scaling) under  `bounded stable orbit equivalence' and `integrable semi-stable orbit equivalence' (a variant of Measure Equivalence) for amenable groups \cite{Aus16'}.

\item Delabie-Koivisto-Ma\^{i}tre-Tessera  prove in \cite{DKMT22} that $l^p$-isoperimetric profile is stable under $L^p$-Measure Equivalence for all $p\geq 1$.  

\item Marrakchi-de la Salle prove in \cite{MdlS23} that the property of admitting an affine isometric action without fixed points on some $L^p$-space or admitting a proper affine isometric action on some $L^p$-space ($1\leq p<\infty$) by a locally compact group is invariant under $L^q$-ME , where $q\geq p$. 

\item  Let G be a right-angled Artin group with $|Out(G)| < \infty$ and let $H$ be a countable group with bounded torsion. In \cite{HH23} Horbez-Huang proves that if there exists an $(L^1, L^0)$-measure equivalence coupling from $H$ to $G$, then $H$ is finitely generated and quasi-isometric to $G$.

\end{itemize}

 However, it has been proved by Gromov in \cite{Gro93} that if two finitely generated groups $\Gamma$ and $\Lambda$ are QI, then their $k$-th $l^p$-cohomology group $l^pH^k(\Gamma)$ and $l^pH^k(\Lambda)$ are isomorphic as a topological  vector space. We prove this result for $k=1$ under $L^q$-ME in our Main Theorem \ref{mainthm:invcoh}.

\subsection{Statement of the main theorem and the corollaries}
 
\begin{thm}\label{mainthm:invcoh}(\textbf{Main Theorem})
Suppose two non-amenable groups $\Gamma$ and $\Lambda$ are $L^q$-ME for some $q\geq 1$. Then $l^pH^1(\Gamma)\neq 0$ if and only if $l^pH^1(\Lambda)\neq 0$, when $1< p\leq q$. If $p=1$, then $l^pH^1(\Lambda)=0$ implies that all affine actions associated to $l^pH^1(\Gamma)$ have bounded orbits. 
\end{thm}

\begin{rem}
By Proposition \ref{lpH1-coh} and Proposition \ref{monotone}, we obtain that  for finitely generated non-amenable groups $\Gamma$, $l^pH^1(\Gamma)\subseteq l^{p'}H^1(\Lambda)$, where $1\leq  p\leq p' \in\mathbb{R}$. 
The quantitity $p_{*}(\Gamma):= inf \{p : l^pH^1(\Gamma)\neq 0\}$ is called the critical exponent of first $l^p$-cohomology of the group $\Gamma$. Using main theorem \ref{mainthm:invcoh}, in particular, we obtain that if two groups $\Gamma$ and $\Lambda$ are $L^q$-ME,  $q> Max\{p_{*}(\Gamma), p_{*}(\Lambda)\}$ and $p_{*}(\Gamma), p_{*}(\Lambda) >1$, , then  $p_{*}(\Gamma)= p_{*}(\Lambda)$. 
\end{rem}

We prove the Main Theorem in Section 5. It is well-known that the conformal dimension of the boundary of a Gromov hyperbolic group  is invariant under quasi-isometry. In this article, we ask the question for $L^q$-ME. We obtain an affirmative answer for a special class of hyperbolic Coxeter groups. We prove this result as a corollary of our main theorem in Section 6. 

\begin{cor} \label{cor:confdim}
Let $\Gamma$ and $\Lambda$ are two hyperbolic  groups with boundaries having Combinatorial Loewner Property. Suppose, $\Gamma$ and $\Lambda$ are $L^q$-ME,  for $q > Confdim (\Gamma)$, $Confdim (\Lambda)$ and $Confdim (\Gamma)$, $Confdim (\Lambda) >1$ . Then, the conformal dimensions (of the canonical conformal gauge) of the groups are equal. 
\end{cor}

Using our main theorem \ref{mainthm:invcoh}, we also prove the following corollaries in Section 6. These corollaries were already known, but we reprove the results here.  

\begin{cor}\label{cor:f2sg}
The free group $F_n$ with $n\geq 2$ and the surface group $\pi_1(\Sigma_g)$ with genus $g\geq 2$ are not $L^1$-ME. 
\end{cor}

\begin{cor}\label{cor:cesg}
Any cocompact lattice in $Isom(\mathbb{H}^3)$  is not $L^{2+\epsilon}$-ME with the central extensions of surface groups for all $\epsilon >0$.
\end{cor}

\subsection{Acknowledgements}

I am very grateful to Prof. Romain Tessera for  valuable discussions on this project. 
I would like to thank Weizmann Institute of Science, Israel and Indian Statistical Institute, Kolkata,  where most of this work has been carried out and National Board for Higher Mathematics, India, for supporting this project financially. 

I would also like to thank Projet ANR-14-CE25-0004 GAMME  for supporting me financially for conferences which helped me in developing knowledge in Geometric Group Theory and Measured Group Theory. I would like to thank Prof. Antoine Gourney for pointing out a missing  reference.

\section{$L^p$-Measure Equivalence}

Two countable discrete groups $\Gamma$ and $\Lambda$ are called \textbf{Measure Equivalent} (ME) if there is a nonzero $\sigma$-finite measure space $(X,\mu)$, which admits commuting free $\mu$-preserving actions of  $\Gamma$ and $\Lambda$  such that  both have finite-measure fundamental domains $X_{\Gamma}$ and $X_{\Lambda}$, respectively, i.e., 

$$ X= \sqcup_{\gamma\in\Gamma}  \gamma X_\Gamma  =  \sqcup_{\lambda\in\Lambda} \lambda X_\Lambda .$$

$(X,\mu)$ is called the coupling space of $\Gamma$ and $\Lambda$. Let $\alpha:\Gamma\times X_\Lambda\rightarrow \Lambda$  
 (resp.\ $\beta:\Lambda\times X_\Gamma\rightarrow \Gamma$) be the corresponding cocycle defined by the rule: for all $x\in X_{\Lambda}$ and 
 $\gamma\in \Gamma$, $\alpha(\gamma, x)\gamma x \in X_\Lambda$ (and symmetrically for $\beta$). We define $\gamma\cdot x:= \alpha(\gamma, x)\gamma x$ for all $\gamma\in\Gamma$ and $x\in X_\Lambda$ and $\lambda\cdot y:= \beta(\lambda, y)\lambda y$ for all $\lambda\in\Lambda$ and $y\in X_\Gamma$. We have the following cocycle relations: 
  $$\alpha(\gamma_1\gamma_2,x)=\alpha(\gamma_1,\gamma_2\cdot x)\alpha(\gamma_2,x)\hspace*{.5cm} \text{and} \hspace*{.5cm}  \beta(\lambda_1\lambda_2,x)=\beta(\lambda_1,\lambda_2\cdot x)\beta(\lambda_2,x)$$

 where $\gamma_1,\gamma_2\in \Gamma$, $x\in X_\Lambda$, $\lambda_1,\lambda_2\in\Lambda$ and $y\in X_\Lambda$. If,  for any $\lambda\in \Lambda$ and  $\gamma\in \Gamma$, the integrals 
 $$\int_{X_\Lambda}|\alpha(\gamma,x)|^p d\mu(x) \hspace*{.5cm} \text{and} \hspace*{.5cm} \int_{X_\Gamma}|\beta(\lambda,x')|^pd\mu(x')$$  are finite, then the groups are called \textbf{$L^p$-measure equivalent}.  The strongest form is when $p=\infty$, in which case the coupling is called {\it uniform}, and the groups {\it uniformly measure equivalent (UME)}, as it generalizes the case of two uniform lattices in a same locally compact group. For $p=1$, the coupling is called {\it integrable}, and the groups are said to be  {\it integrable measure equivalent (IME)}.

\vspace{0.5cm}

\begin{rem}
Measure Equivalence is closely related to the concept of Orbit Equivalence (OE) in ergodic theory. Suppose $\Gamma$ and $\Lambda$ are Measure Equivalence with $X_\Gamma=X_\Lambda$. Then the groups are orbit equivalent. 
Moreover, two groups are ME if and only if they are `weakly OE' or `stably OE'
\end{rem}

\section{The first $l^p$-cohomology of groups}

\vspace{0.5cm}

\subsection{Definition of $l^p$-cohomology} 

We fix $p\in [1,\infty)$. Let $\Gamma$ be a finitely generated group and  fix a Cayley graph of $\Gamma$. Suppose $V_\Gamma$ and $E_\Gamma$ are the vertex set and the edge set of the Cayley graph, respectively.

``The first $l^p$-cohomology group of $\Gamma$'', denoted by $l^p H^1(\Gamma)$, will be defined as a quotient of the space of functions $f : V_\Gamma \rightarrow \mathbb{R}$.   We follow the definition given in \cite{MT10}. 

Let $l^p(V_\Gamma)$ (respectively $l^p(E_\Gamma)$) denote the collection of $p$-summable functions
$f : V_\Gamma \rightarrow \mathbb{R}$ (respectively $g : E_\Gamma\rightarrow \mathbb{R}$).
 
The differential of  $f : V_G \rightarrow \mathbb{R}$, denoted $df$, is the map from $E_\Gamma$ to $\mathbb{R}$ defined by $df(ab) = f(b) - f(a)$. Since $V_\Gamma$ has bounded valence, $df \in l^p(E_\Gamma)$,  whenever $f \in l^p(V_\Gamma)$. Since $V_\Gamma$ is connected, $df=0$ if and only if $f$ is a constant.

\begin{defn}
The ``first $l^p$ cohomology group'' of $\Gamma$ written $l^pH^1(\Gamma)$, is defined as follows:  $ \{f :V_\Gamma\rightarrow \mathbb{R} : df \in l^p(E_\Gamma)\}/ (l^p(V_\Gamma)\bigoplus\mathbb{R}\cdot \textbf{1})$ ,
where $\textbf{1}$ denotes the constant function: $\textbf{1}(v) = 1$ for all $v\in V_\Gamma$. 
\end{defn}

Similarly, we define ``first reduced $l^p$ cohomology group'' $\Gamma$. 

\begin{defn}
The ``first reduced $l^p$ cohomology group'' of $\Gamma$ written $l^p\bar{H}^1(\Gamma)$, is defined as follows:  $ \{f :V_\Gamma\rightarrow \mathbb{R} : df \in l^p(E_\Gamma)\}/ \overline{(l^p(V_\Gamma)\bigoplus\mathbb{R}\cdot \textbf{1})}$ ,
where $\textbf{1}$ denotes the constant function: $\textbf{1}(v) = 1$ for all $v\in V_\Gamma$ and the closure is with respect to the topology of point-wise convergence. 

\end{defn}
 
 We remark that the isomorphism type of $l^pH^1(\Gamma)$ is a quasi-isometry invariant of the underlying graph $V_\Gamma$, and thus an invariant of $\Gamma$. We will write $l^pH^1(\Gamma)$  for the isomorphism class of first $l^p$-cohomology groups associated to $\Gamma$. 
 
 \vspace{0.5cm}

\subsection{List of some known results of first $l^p$-cohomology of discrete groups:}
 \begin{itemize}
 \item Suppose $A$ and $B$ are two finitely generated groups with $|A|\geq 2$ and $|B|\geq 3$. Then 
 $\Gamma=A\star B$ satisfies $l^pH^1(\Gamma)\neq 0$ for all $p\geq 1$ (see \cite{Bou10}, Proposition 1.1). In particular, for the free group with finite number generators, the first $l^p$-cohomology is non-zero for all $p\geq 1$ . 
 \item Suppose $\Gamma$ is a cocompact lattice in $Isom(\mathbb{H}^n)$, where $n\geq 2$. 
 Then $l^pH^1(\Gamma)= 0$ iff $p\in [1,n-1]$. (see \cite{Pan89'}, Proposition 2.1)
 \item Let $\Gamma$ be a word-hyperbolic group. Then for $p$ large enough we have $l^pH^1(\Gamma)\neq 0$ (see  \cite{Bou10} Proposition 2.3, \cite{Yu05}). 
 \end{itemize}
 
 \vspace{0.5cm}
 
\subsection{1-cohomology with coeffcients in a Banach Space}

Let $\Gamma$ be a discrete group. Let $\pi$ be an isometric representation of $\Gamma$ on a Banach space $B$. 

\begin{itemize}
\item[1.] A mapping $b:\Gamma \rightarrow B$ s.t.  $b(gh) = b(g) + \pi(g)b(h)$, for all $g, h \in \Gamma$
is called a \textbf{1-cocycle} with respect to $\pi$.
\item[2.] A 1-cocycle $b : \Gamma \rightarrow B$ for which there exists $\xi\in B$ such that
$b(g)=\pi(g)\xi - \xi$, for all $g\in \Gamma$, is called a \textbf{1-coboundary} with respect to $\pi$. 
\item[3.]  The quotient space $H^1(\Gamma,\pi) = Z^1(\Gamma,\pi)/B^1(\Gamma,\pi)$ is called 
\textbf{1-cohomology} associated to the representation $\pi$. 
The \textbf{reduced 1-cohomology} of $\Gamma$ associated to the representation $\pi$ is the quotient vector space $H^1(\Gamma,\pi) = Z^1(\Gamma,\pi)/\overline{B}^1(\Gamma,\pi)$, where the closure is taken w.r.t. the pointwise convergent topology. Sometimes, we use the notations 
$H^1(\Gamma,B)$ and  $\overline{H}^1(\Gamma, B)$ for 1-cohomology and reduced 1-cohomology, respectively. 
\end{itemize}

The following proposition relates the first $l^p$-cohomology of a group $\Gamma$ with the 1-cohomology with coefficients in $l^p(\Gamma)$.  

\begin{prop}\label{lpH1-coh}  $l^pH^1(\Gamma)=H^1(\Gamma, l^p(\Gamma))$, where $\rho_\Gamma$ is acting on $l^p(\Gamma)$ by  right regular action. Similarly, $l^p\overline{H}^1(\Gamma)=\overline{H}^1(\Gamma, l^p(\Gamma))$, where $p\geq 1$,
\end{prop} 
\begin{proof}
We briefly sketch the proof. Let $f:V_\Gamma \rightarrow \mathbb{R}$ be a representative of a  class in $l^pH^1(\Gamma)$.
Suppose $S$ is a generating subset of $\Gamma$.  
We define $b:\Gamma\rightarrow l^p(\Gamma)$ by $b(s)(\gamma)=f(v_{\gamma s})-f(v_{\gamma})$. On the other hand, if 
we have a cocycle $b:\Gamma \rightarrow l^p(\Gamma)$, we define $f(e_{\gamma,\gamma s})= b(s)(\gamma)$, where 
 $e_{\gamma,\gamma s}$ is the edge between $v_{\gamma s}$ and $v_\gamma$ indexed by $s$. 
\end{proof}

We remark that the 1-cohomology of a group $\Gamma$ with coefficients in $l^p(\Gamma)$ is monotone in the following sense. 

\begin{prop}(\cite{Loh90},\cite{Pul06} Proposition 6.2) \label{monotone}
 Let $\Gamma$ be a finitely generated non-amenable group. If $1 < p \leq p' \in\mathbb{R}$, then $H^1\big(\Gamma, l^p(\Gamma)\big) \subseteq H^1\big(\Gamma, l^{p'}(\Gamma)\big)$
\end{prop}

Using Proposition \ref{monotone} and Proposition \ref{lpH1-coh}, we obtain that the first $l^p$-cohomology is monotone (i.e.,  if $p\leq p'$,  $l^pH^1(\Gamma)\subseteq l^{p'}H^1(\Gamma))$ . Now, we mention Theorem \ref{lpbarlp} and Proposition \ref{dirint} which will be useful in the proof of our main theorem \ref{mainthm:invcoh}.

\begin{thm}\label{lpbarlp}
 Suppose  $p\in (1,\infty)$. TFAE: 
 \begin{itemize}
\item[(a)] $\Gamma$ is non-amenable,
\item[(b)] $l^pH^1(\Gamma) = l^p\overline{H}^1(\Gamma)$,
\end{itemize}
\begin{proof}
We refer the readers to \cite{Bou10} ( Theorem 1.5, p. 9 ) for the proof of the proposition. 
\end{proof}
\end{thm}

\begin{prop}\label{dirint}
 If $\overline{H}^1(\Lambda, \pi)= 0$, then $\overline{H}^1\big(\Lambda, \int_X^{\oplus_p} \pi^{(x)} d\mu(x)\big)= 0$ where $\Lambda$ is a discrete group, $\pi^{(x)}=\pi$ for all $x$ and $\pi$ is the left-regular or right-regular representation of $\Lambda$ on $l^p(\Lambda)$, where $1\leq p <\infty$. 
\end{prop}
\begin{proof}
 We refer the readers to  Theorem \ref{dirint1} of the Appendix of this article for the proof of this proposition. 
\end{proof}

 \section{Induced representation and Induced affine action}
  
\subsection{Induced representation}  
  
We assume that $\Gamma$ and $\Lambda$ are Measure Equivalent (not necessarily $L^p$-ME) as in Section 2.
 Given $1\leq p\leq \infty$ and some isometric representation $\pi$ of $\Gamma$ on a Banach space $\mathcal{B}$, we define the induced representation $\Ind_\Gamma^{p,\Lambda}\pi$ to $\Lambda$ on
$$L^p(X_\Gamma,\mathcal{B}):=\{\psi:X_\Gamma\rightarrow \mathcal{B} \hspace*{.2cm}  |  \int_{X_\Gamma}|\psi|^p d\mu<\infty\} \hspace*{.2cm} \psi \mbox{ is Bochner-measurable}$$
in the following way:
$$\lambda\psi(x)=\pi(\beta(\lambda^{-1},x)^{-1})\psi(\lambda^{-1}\cdotp x).$$
We refer the reader to the appendix for the definition of `Bochner-measurable function'. Observe that this is a ``linear induction", and no integrability assumption on the coupling is required. 
When $p$ is clear from the context, we denote the induced representation by $\Ind_\Gamma^{\Lambda}\pi$.

  \vspace{0.5cm}

  \subsection{Induced cocycle in the induced representation}

 Let $b:\Gamma\rightarrow \mathcal{B} $ be a 1-cocycle coming from $Z^1(\Gamma, \mathcal{B})$, where $\Gamma$ is acting on $\mathcal{B}$ by $\pi$ as in the previous section. We consider the induced representation $\Ind_\Gamma^{p,\Lambda}\pi$ of $\Lambda$ with coefficients in $L^p(X_\Gamma, \mathcal{B})$.  We assume that $\Gamma$ and $\Lambda$ are $L^q$-ME, where $q\geq p$.

 We define 1-cocycle $B\in Z^1(\Lambda, L^p(X_\Gamma, \mathcal{B})$ associated to $\Ind_\Gamma^{p,\Lambda}\pi$ by the following way: 

\begin{equation}
   B(\lambda)(y):= b[\beta(\lambda^{-1}, y)^{-1}],  
\end{equation}   

where $\lambda\in \Lambda$, $y\in X_\Gamma$.

\begin{prop}
 $B$ is well-defined, i.e., $B$ is $p$-integrable, and $B$ is a cocycle. 
\end{prop}
\begin{proof}
We fix $\lambda\in\Lambda$ and we suppose that $\Gamma$ is generated by a symmetric generating set $S=\{s_1, \dots, s_k\}$. 

\begin{align*}
\int_{X_\Gamma} \| B(\lambda)(y)\|^p d\nu(y) 
  & = \int_{Y} \| b[\beta(\lambda^{-1}, y)^{-1}] \|^p d\nu(y) \\
  & \leq  max \{\|b(s_i)\|^p\}_{i=1}^k \int_{X_\Gamma} \mid \beta(\lambda^{-1}, y)\mid ^p d\nu(y) <\infty  
\end{align*}  
In the above inequality, we use the properties  of 1-cocycle $b$ and $p$-integrability of $\beta$. Now, we check that $B$ is a 1-cocycle. We observe that $\{B(\lambda)\}(y)=b[\{\beta(\lambda^{-1},y)\}^{-1}]=b[\beta(\lambda,\lambda^{-1}\cdot y)]$ using the cocycle relation for $\beta$ . Therefore, 

\begin{align*}
\{B(\lambda_1\lambda_2)\}(y)
& = b[\beta(\lambda_1\lambda_2,\lambda_2^{-1}\lambda_1^{-1}\cdot y)]\\
& = b[\beta(\lambda_1,\lambda_2\lambda_2^{-1}\lambda_1^{-1}\cdot y)\beta(\lambda_2, \lambda_2^{-1}\lambda^{-1}\cdot y)]\\
& = b[\beta(\lambda_1,\lambda_1^{-1}\cdot y)]+\pi\big(\beta(\lambda_1,\lambda_1^{-1}\cdot y)\big) b\big(\beta(\lambda_2,\lambda_2^{-1}\lambda_1^{-1}\cdot y)\big)\\
& = \{B(\lambda_1)\}(y)+ \pi\big(\beta(\lambda_1,\lambda_1^{-1}\cdot y)\big) [\{B(\lambda_2)\}(\lambda_1^{-1} \cdot y)]\\
& = \{B(\lambda_1)\}(y)+ [\{\Ind_\Gamma^\Lambda\pi(\lambda_1)\}B(\lambda_2)](y)
\end{align*}

Hence, $B\in Z^1\big(\Lambda, L^p(X_\Gamma, \mathcal{B})\big)$. 
\end{proof}

\subsection{Induced representation from right-regular action}  

 We define a measure preserving isomorphism $\psi:\Lambda\times X_\Lambda \rightarrow \Gamma\times X_\Gamma$ by
  $$(\lambda',x')\mapsto \big(\gamma' \beta(\lambda',y')^{-1},\lambda'\cdot y'\big),$$
  where $x'= \gamma'y'$, $\gamma'\in\Gamma$ and $y'\in X_\Gamma$. It is easy to see that $\Psi$ induces an isometric isomorphism $\Psi: L^p(X_\Gamma,l^p(\Gamma))\rightarrow L^p(X_\Lambda,l^p(\Lambda))$, for $p\geq 1$.
  
  \begin{lem}\label{indright}
   $\Psi\circ[Ind_\Gamma^\Lambda \rho_\Gamma(\lambda)]=[\int_{X_\Lambda}^{\oplus_p}\tau_\Lambda(\lambda)]\circ \Psi$, where $\rho_\Gamma$ is the right-regular representation of $\Gamma$ and $\tau_\Lambda$ is the left-regular representation of $\Lambda$.
 \end{lem}
  \begin{proof}
  For $x'\in X_\Lambda$ and $\lambda\in\Lambda$, we have
  \begin{align*}
  [\{\Psi\circ Ind_\Gamma^\Lambda \rho_\Gamma(\lambda)f\}(x')](\lambda')
  & = [\{Ind_\Gamma^\Lambda \rho_\Gamma(\lambda)f\}(\lambda'\cdot y')]\big(\gamma'\beta(\lambda',y')^{-1}\big)\\
   & =\{ f(\lambda^{-1}\lambda'\cdot y')\}\big(\gamma'\beta(\lambda',y')^{-1}\beta(\lambda^{-1},\lambda'\cdot y')^{-1}\big)\\
  & = \{ f(\lambda^{-1}\lambda'\cdot y')\} \big(\gamma' \beta(\lambda^{-1}\lambda',y')^{-1}\big)\\
  & = [\{\big(\int_{X_\Lambda}^{\oplus_p}\tau_\Lambda(\lambda)\circ \Psi\big)(f)\}(x')](\lambda')
 \end{align*}

Hence, we have our lemma. 
  \end{proof}

\subsection{Induced affine action} 

As in Subsection 4.2, we assume that the 1-cocycle $B$ is  induced from the 1-cocycle $b\in Z^1(\Gamma,\mathcal{B})$ and associated to the induced representation $\Ind_\Gamma^\Lambda \pi$. We denote the affine action of $\Gamma$ associated to $b$ by $\eta$ and the affine action of $\Lambda$ on $L^p(X_\Gamma,\mathcal{B})$ associated to $B\in Z^1(\Lambda, \Ind_\Gamma^\Lambda \pi)$ by $\Ind_\Gamma^\Lambda \eta$. 

We now consider the setting in the proof of Theorem 7.4 in \cite{MdlS23}.  Since the coupling space $X$ is equal to $\Gamma  X_\Gamma$,   we can define $p_1:X\rightarrow \Gamma$ by $w=\gamma x\mapsto \gamma$ for all $w=\gamma x \in X$, where $\gamma\in\Gamma$ and $x\in X_\Gamma$.  
Let $L_0(X,\mu, \mathcal{B})^{\Gamma}$ be the set of all $\Gamma$-equivariant Bochner-measurable maps from $Z$ to $\mathcal{B}$, where $\Gamma$ acts on $\mathcal{B}$ by affine action $\eta$ (we refer the readers to the second paragraph of the Appendix for the definition of Bochner measurable function.  We define $f_0\in L_0(X,\mu, \mathcal{B})^{\Gamma}$ by $f_0(w)=p_1(w)\cdot 0=b(\gamma)$, where $w=\gamma x$, $\gamma\in\Gamma$ and $x\in X_\Gamma$. Now, we define an affine subspace 

$$F=\{f\in  L_0(X,\mu, \mathcal{B})^{\Gamma} : \|f-f_0\|_{p,\Gamma}<\infty \}.$$

We further define $\Phi: F\rightarrow L^p(X_\Gamma,\mathcal{B})$ by $f\mapsto (f-f_0)\vert_{X_\Gamma}=f\vert_{X_\Gamma}$. There is the following natural affine action $\zeta$ of $\Lambda$ on $F$: $[\zeta(\lambda)f](w):=f(\lambda^{-1}w)$ for all $\lambda\in \Lambda$ and $f\in F$. On the other hand, there is the affine action $\Ind_\Gamma^\Lambda \eta$ of $\Lambda$ on $L^p(X_\Gamma,\mathcal{B})$. 

\begin{lem}\label{equiaffine}
$\Phi$ is equivariant under the affine action $\zeta$ of $\Lambda$ on $F$ and the affine action $\Ind_\Gamma^\Lambda \eta$ of $\Lambda$ on $L^p(X_\Gamma,\mathcal{B})$.
\end{lem}
\begin{proof}

For all $x\in X_\Gamma$, 

\begin{align*}
\{\zeta(\lambda)f\}(x)
& = f (\lambda^{-1} x)\\
& = f [\beta(\lambda^{-1},x)^{-1}\big( \lambda^{-1}\cdot x\big)]\\
& =\beta(\lambda^{-1},x)^{-1}\cdot f(\lambda^{-1}\cdot x)\\
& = \pi\big(\beta(\lambda^{-1},x)^{-1}\big) f(\lambda^{-1}\cdot x)+b\big(\beta(\lambda^{-1},x)^{-1}\big)\\
& = [\Ind_\Gamma^\Lambda \eta(\lambda) f](x)
\end{align*}
Hence, we have our lemma. 
\end{proof}

 \vspace{0.5cm}

\section{Proof of the main theorem }\label{proof}

In this section, we prove our main theorem \ref{mainthm:invcoh}. 

\begin{proof}

  Suppose $\Gamma$ and $\Lambda$ are $L^q$-ME,  where $q \geq 1$.  Let $(X,\mu)$ be the coupling space, and $X_\Gamma$ and  $X_\Lambda$ be fundamental domains 
of $\Gamma$ and $\Lambda$ respectively. Moreover, we assume that  $\alpha:\Gamma \times X_\Lambda \rightarrow \Lambda$ and $\beta:\Lambda \times X_\Gamma \rightarrow \Gamma$ are $L^q$-integrable cocycles.   Suppose  $l^p H^1(\Gamma)\neq 0$, where $q\geq p\geq 1$.  We prove that $l^pH^1(\Lambda)\neq 0$. From Step 1 to Step 3, we prove our theorem for $1< p <\infty$, and in Step 4 we prove our theorem for $p=1$.

\vspace{5mm}
 
 \textbf{Step 1:}
 
 \vspace{5mm}
 
By Proposition \ref{lpH1-coh}, we obtain that $H^1(\Gamma, l^p(\Gamma))=l^p H^1(\Gamma)\neq 0$, $\Gamma$ acts on $l^p(\Gamma)$ by right-regular action $\rho_\Gamma$. Now, we will prove that $H^1\big(\Lambda, Ind_\Gamma^\Lambda(\rho_\Gamma)\big)\neq 0$. This proof is a particular case of Theorem 7.4 (i) in  \cite{MdlS23}.  For the sake of completeness, we give a proof of this result here. Let $b\in Z^1\big(\Gamma, l^p(\Gamma)\big)$ is a non-zero cocycle class in $H^1(\Gamma, l^p(\Gamma))$. We use the notations use in the Subsections 4.1, 4.2, 4.3 and 4.4 (the Banach space $\mathcal{B}$ is replaced by $l^p(\Gamma)$ and the representation $\pi$ is replaced by $\rho_\Gamma$). We assume by contradiction that $H^1\big(\Lambda, Ind_\Gamma^\Lambda(\rho_\Gamma)\big)= 0$. We consider the affine action associated to the representation $\Ind_\Gamma^\Lambda \rho_\Gamma$ and cocycle $B$.  By our assumption, this action has a fixed point in $L^p(X_\Gamma, l^p(\Gamma))$. By Lemma \ref{equiaffine}, the affine action $\zeta$ of $\Lambda$ has a fixed point in $F$. Let $f\in F$ be a fixed point for $\zeta$. It implies that $f: X_\Lambda\rightarrow l^p(\Gamma)$ is a $\Gamma$-equivariant measurable map which is also $\Lambda$-invariant. 
Therefore, we can see $f$ as a $\Gamma$-equivariant map from $X_\Lambda$ to $l^p(\Gamma)$. Since $X_\Lambda$ admits a $\Gamma$-invariant probability measure $\mu_\Lambda$, we can push it forward by $f$ and obtain a $\Gamma$-invariant measure on $l^p(\Gamma)$. By Lemma 2.14 in \cite{BFGM07},  when $1< p <\infty$ , we obtain that $\eta$ has a fixed point in $l^p(\Gamma)$, which is a contradiction.

\vspace{5mm}
 
\textbf{Step 2:} 
 
\vspace{5mm}

By Lemma \ref{indright}, we obtain that $\Psi\circ[Ind_\Gamma^\Lambda \rho_\Gamma(\lambda)]=[\int_{X_\Lambda}^{\oplus_p}\tau_\Lambda(\lambda)]\circ \Psi$, for all $\lambda\in\Lambda$. Since $\Lambda$ is non-amenable, by Lemma 2 of \cite{BMV05}  $\int_{X_\Lambda}^{\oplus_p} \tau_\Lambda$ does not have almost invariant vectors. By Th\'{e}or\`{e}me 1 in \cite{Gui72}, we have the following result: Let $G$ be a locally compact group and $\Gamma \curvearrowright_\pi B$ be an isometric $G$-representation. Assume that $\pi$ does not weakly contain the trivial representation. Then $\overline{H}_{ct}^1(G,\pi)=H_{ct}^1(G,\pi)$. 
Therefore,  we obtain that 

\begin{align*}
H^1\big(\Lambda, Ind_\Gamma^\Lambda \rho_\Gamma, L^p(X_\Gamma, l^p(\Gamma))\big)
  & = H^1\big(\Lambda, \int_{X_\Lambda}^{\oplus_p}\tau_\Lambda, L^p(X_\Lambda, l^p(\Lambda))\big)\\
  & \overline{H}^1\big(\Lambda, \int_{X_\Lambda}^{\oplus_p}\tau_\Lambda, L^p(X_\Lambda, l^p(\Lambda))\big). 
\end{align*}

\vspace{5mm}

\textbf{Step 3:}

\vspace{5mm}

 Since $\overline{H}^1\big(\Lambda, \int_{X_\Lambda}^{\oplus_p}\tau_\Lambda, L^p(X_\Lambda, l^p(\Lambda))\big) \neq 0$, using Proposition \ref{dirint},  we obtain that $\overline{H}^1(\Lambda, \tau_\Lambda)\neq 0$. Now, we define $\chi: l^p(\Lambda) \rightarrow l^p(\Lambda)$ by extending the map $\lambda\mapsto \lambda^{-1}$ for all $\lambda\in\Lambda$. Observe that $\chi$ is $\Lambda$-equivariant, where $\Lambda$ acts on the domain and the range by left-regular representation $\tau_\Lambda$ and right-regular representation $\rho_\Lambda$, respectively.  Since $\chi$ is a $\Lambda$-equivariant invertible isometry, we have  $\overline{H}^1(\Lambda, \tau_\Lambda)=\overline{H}^1(\Lambda, \rho_\Lambda)$.  Finally, using Proposition  \ref{lpH1-coh} and Theorem \ref{lpbarlp}, we have  $l^pH^1(\Lambda)\neq 0$ for $1< p <\infty$. 
 
 \vspace{5mm}
 
 \textbf{Step 4:}
 
 \vspace{5mm}
 
 Now, we deal with case $p=1$. Since Step 2 and Step 3 hold for any $p\geq 1$, using these two steps we obtain that $H^1\big(\Lambda, Ind_\Gamma^\Lambda(\rho_\Gamma)\big)\neq 0$ implies 
 \linebreak
 $\overline{H}^1\big(\Lambda, \int_{X_\Lambda}^{\oplus_p}\tau_\Lambda, L^p(X_\Lambda, l^p(\Lambda))\big) \neq 0$, and which further implies that 
 $l^pH^1(\Lambda)\neq 0$. Therefore, contra-positively,  $l^pH^1(\Lambda)= 0$ implies that $H^1\big(\Lambda, Ind_\Gamma^\Lambda(\rho_\Gamma)\big)= 0$. Now, using Step 1, we obtain that there exists a $\Gamma$-invariant measure on $l^p(\Gamma)$. Since the argument of the proof of Lemma 2.14 $ (4)\Rightarrow (1)$  in \cite{BFGM07} holds for $p=1$,  we obtain that $\eta$ has bounded orbits in $l^p(\Gamma)$. Hence we have our theorem for $p=1$. 

\end{proof}

\section{Applications}

 \subsection{Conformal dimension and first $l^p$-cohomology of groups}

 There is an excellent result due to  Bourdon and Pajot \cite{BP03}, which relates geometric quantity `conformal dimension' of the boundary of a class of hyperbolic groups and  `algebro-analytical quantity' the first $l^p$-cohomology' of this class of groups. 

\begin{thm}\cite{BP03}(Bourdon-Pajot, Corollaire 0.4)\label{lpcohconfd}
Let $\Gamma$ be  a hyperbolic  group without torsion and with boundaries having  `combinatorial Loewner property' (CLP). Then  $$Confdim{\Gamma}= inf\{p\neq 0 : l^pH^1(\Gamma)\neq 0\}$$
\end{thm}

Now, using the above theorem and Theorem \ref{mainthm:invcoh}, we obtain the following corollary: 

\begin{cor} \label{cor:confdim}
Let $\Gamma$ and $\Lambda$ are two hyperbolic  groups without torsion and with boundaries having Combinatorial Loewner Property. Suppose, $\Gamma$ and $\Lambda$ are $L^q$-ME,  for $q > Confdim (\Gamma)$, $Confdim (\Lambda)$ and  $Confdim (\Lambda)$ and $Confdim (\Gamma)$, $Confdim (\Lambda) >1$ Then, the conformal dimensions (of the canonical conformal gauge) of the groups are equal. 
\end{cor}

\subsection{Free groups and surface groups are not $L^1$-Measure Equivalent}
Let $F_n$ be the free group with finitely many generators $n\geq 2$ and $\Sigma_g$ be the surface group with genus $g\geq 2$. We briefly show that there exists an affine action of $F_n$ on $l^1(F_n)$ (where the linear part is the left-regular representation on $l^1(F_n)$) which has unbounded orbits. We consider the Cayley graph $\mathcal{G}$ of $F_n$ and we denote the identity element in the Cayley graph as `$o$'. We define a 1-cocycle $b:F_n\rightarrow l^1(F_n)$ as follows: $b(g)(h)=1$ if the path from $g$ to $h$ directs towards $o$ and $b(g)(h)=0$ otherwise. In other words, $b(g)=\mathbf{1}_{\overline{og}}$, where $\bar{og}$ is the set of points in the line segment connecting $o$ and $g$. It is easy to see that $b$ is a 1-cocycle. Since $\| b(g)\|=d(o,go)$, we obtain that $\| b(g)\|\rightarrow\infty$ as $|g|\rightarrow\infty$. Therefore, the affine action associated to the 1-cocycle $b$ has unbounded orbits. Now, since  $l^1H^1\big(\pi_1(\Sigma_g)\big)=0$ (see Subsection 3.2), using Main Theorem \ref{mainthm:invcoh}, we obtain the following corollary:

\begin{cor}\label{cor:f2sg}
The free group $F_n$ with $n\geq 2$ and the surface group $\pi_1(\Sigma_g)$ with genus $g\geq 2$ are not $L^1$-ME. 
\end{cor}

This corollary has been proved by different method in Lemma 5.4 of \cite{BFS13}.

\subsection{The fundamental group of 3-manifolds and $L^q$-Measure Equivalent}

In this subsection, we discuss $L^q$-ME between the fundamental groups of 3-manifolds (the extensions of surface groups and integers) with three geometries $\mathbb{H}^3$, $\mathbb{H}^2\times\mathbb{R}$ and $\widetilde{SL_2(\mathbb{R})}$ of Thurston's eight geometries. Let $\Gamma$ be the fundamental group of a closed 3-manifold $M$, with constant sectional curvature  -1, and which is a fiber bundle over the circle and whose fibers are a closed surface of genus at least 2. Since $\Gamma$ is a cocompact lattice in $Isom(\mathbb{H}^3)$, $l^pH^1(\Gamma)=0$ iff $p\in [1,2]$. Now, we consider the non-trivial central extension $\widetilde{\pi_1(\Sigma_g})$ of surface group $\pi_1(\Sigma_g)$ and the trivial central extension $\pi_1(\Sigma_g)\times \mathbb{Z}$. We know from \cite{DT16} (Theorem 1.1) that $\widetilde{\pi_1(\Sigma_g)}$ and $\pi_1(\Sigma_g)\times \mathbb{Z}$ are not $L^1$-ME. Since the center of $\widetilde{\pi_1(\Sigma_g)}$ and $\pi_1(\Sigma_g)\times \mathbb{Z}$ are infinite, by using Proposition 1.9 in \cite{Bou10} we have the first $l^p$-cohomology of these two groups are zero for all $p\in (1,\infty)$. Now, using Theorem \ref{mainthm:invcoh}, we obtain the following corollary: 

\begin{cor}\label{cor:cesg}
Any cocompact lattice in $Isom(\mathbb{H}^3)$  is not $L^{2+\epsilon}$-ME with $\widetilde{\pi_1(\Sigma_g)}$ (QI to $\widetilde{SL_2(\mathbb{R})}$) or $\Sigma_g\times \mathbb{Z}$ (QI to $\mathbb{H}^2\times\mathbb{R}$) for all $\epsilon >0$. In particular, $\widetilde{\pi_1(\Sigma_g)}$ (QI to $\widetilde{SL_2(\mathbb{R})}$) or $\pi_1(\Sigma_g)\times \mathbb{Z}$ (QI to $\mathbb{H}^2\times\mathbb{R}$) are not $L^{2+\epsilon}$-ME with $\Gamma$ for all $\epsilon >0$, where $\Gamma$ is the cocompact lattice $Isom(\mathbb{H}^3)$ as defined above. 
\end{cor}

This corollary also follows from Theorem A of \cite{BFS13}. 

\section{Questions}

We can define higher $l^p$-cohomology groups $l^pH^k(\Gamma)$ of a group $\Gamma$ for any $k\in\mathbb{N}$ (see Section 1.3 in \cite{Bou10} for the details).  We do not know whether our main thereom \ref{mainthm:invcoh} is true for higher $l^p$-cohomology groups.

   \vspace{5mm}
   
 \textbf{Question 1}: Suppose two non-amenable groups $\Gamma$ and $\Lambda$ are $L^q$-ME for some $q>1$. Is it true that $l^pH^k(\Gamma)\neq 0$ if and only if $l^pH^k(\Lambda)\neq 0$, where $1< p\leq q$ and $k>1$. 
 
   \vspace{5mm}

We study Corollary \ref{cor:confdim} in the context of hyperbolic  groups  with CLP.  We do not know whether it is true  without the assumption of CLP or more generally for all hyperbolic groups. 
   
   \vspace{5mm}
   
 \textbf{Question 2}: Is the theorem true for all hyperbolic groups? 
 
    \vspace{5mm}

\section{Appendix}

In the literature, it has not been written much so far on the group cohomology with coefficients in a direct integral of Banach spaces (in particular for $l^p$-spaces). We give a short introduction of this topic and we prove Proposition \ref{dirint}.

We consider a locally compact second countable group $G$ with Haar measure $\nu$ and a measurable field of strongly  continuous representations on a Banach  space $(E, \|\cdot\|)$ parametrized by the measure space $(X,\mu)$ and  denoted by $\{\tau_G^{(x)} : x\in X\}$.  We define the direct integral of  the representations $\tau_\Lambda^{(x)}$ over the measure space $(X,\mu)$, denoted by $\int_X^{\oplus_p} \tau_\Lambda^{(x)} d\mu(x)$, using `Bochner integral'. We first define `Bochner Measurable' function.  The function $s: X\rightarrow E$ defined as $s=\sum_{i=1}^k \chi_{X_i} \xi_i$ (where $X_i$'s are disjoint measurable subsets of $X$ and $\xi_i$'s are vectors in $E$) is called a `simple function'. Now, a function $f: X\rightarrow E$ is called a Bochner Measurable function if there exists a sequence of simple functions $\{s_j\}_{j=1}^\infty$ such that $\|f(x)- s_j(x) \|\rightarrow 0$ as $j\rightarrow \infty$ for almost all $x\in X$.  Now, we define $\int_X^{\oplus_p} \tau_\Lambda^{(x)} d\mu(x)$ in the following way: 

$$  \int_X^{\oplus_p} \tau_G^{(x)} d\mu(x):=\{f: X \rightarrow E: \int_X \| f(x) \|^p   d\mu(x) <\infty \hspace*{1mm}\mbox {and f is Bochner measurable} \}, $$

where  $G$ acts on the above-mentioned space in the canonical way.  We prove the following proposition.

\begin{thm}\label{dirint1}
If  $\overline{H}^n(G, \tau_G^{(x)})= 0$ for a.e. $x$, then $\overline{H}^n\big(G, \int_X^{\oplus_p} \tau_\Lambda^{(x)} d\mu(x)\big)= 0$. In particular,  $\overline{H}^1(\Lambda, \pi)= 0$ implies that $\overline{H}^1\big(\Lambda, \int_X^{\oplus_p} \pi^{(x)} d\mu(x)\big)= 0$ where $\Lambda$ is a discrete group, $\pi^{(x)}=\pi$ for all $x$ and $\pi$ is the left-regular or right-regular representation of $\Lambda$ on $l^p(\Gamma)$ where $1\leq p <\infty$. 
\end{thm}

We first introduce  some notations and definitions.  Suppose $p\in [1,\infty)$  We denote by $L^p_{loc}(G^n,E)$ the set of measurable functions $f$ from $G^n$ to $E$ such that $f\vert_K \in  L^p(K,E)$ for all $K\in K(G^n)$, where we consider the product measure $\nu^n:=\nu\times\cdots\times\nu$ ($n$-times) and $K(G^n)$ is the collection of compact subsets of $G^n$. One gives the topology defined by the semi-norms

$$ P_K(f) = \big( \int_K \| f(x) \|^p d\nu^n(x) \big )^{1/p}$$

We denote by $L^p_C(G^n, E)$   the space of $p$-integrable functions with compact support. This space  is an inductive limit of the spaces $L^p(K,E)$,  where $K\in K(G^n)$.   We  define  $\mathcal{C}(G^n,E)$ as the collection of continuous functions $f$ from $G$ to $E$ with compact convergence topology. 

The reduced cohomology defined by the following chain is denoted by $\overline{H}^{n,p}\big(G, E\big)$:

$$ 0 \longrightarrow  L^p_{loc}(G,E) \longrightarrow L^p_{loc}(G^2,E) \longrightarrow \cdots  $$

\begin{lem}\label{coh-hom}
The dual of $L^p_{loc}\big(G^n,E\big)$ can be identified algebraically with $L_C^q\big(G^n, E'\big)$ by the following duality: 
$$ <\phi, \psi>:= \int_{G^n} <\phi (x), \psi(x)> d\nu^n(x),$$
\end{lem}
where  $1/p+1/q=1$,   $E'$ is the dual space of $E$ and $<\phi (x), \psi(x)> = \psi(x) \big(\phi (x) \big)$ for all $x \in G^n$.  We use Hahn-Banach Theorem to conclude the above lemma.

By the above duality,  we obtain the following complex: 

$$\cdots  \longrightarrow L_C^q\big(G^2, E'\big) \longrightarrow L_C^q\big(G, l^q(\Lambda)\big)\longrightarrow 0, $$   where $1/p+1/q=1$. 

\vspace{5mm}

Proof of Theorem \ref{dirint}:

\vspace{5mm}

\textit{Step 1}:   We claim that :  if  $\overline{H}^{n,p}(G, \tau_G^{(x)})= 0$ for a.e. $x$, ,  then $\overline{H}^{n,p}\big(G, \int_X^{\oplus_p} \tau_G^{(x)} d\mu(x)\big)= 0$.  We use the same argument given in Proposition 2.6 (page 190, \cite{Gui80}) to prove the claim, where use the duality between  $L^p_{loc}\big(G^n,E\big)$ and $L_C^q\big(G^n, E'\big)$ as given in Lemma \ref{coh-hom}. 


\textit{Step 2}:  Using property (i) of D.2.2 (p. 340, \cite{Gui80}), we obtain that $\mathcal{C}(G^n,E)$ can be canonically identified with a dense subset of $L^p_{loc}(G^n, E)$.  Hence the $n$-th reduced cohomology w.r.t. the chain $\mathcal{C}(G^n,E)$, i.e., $\overline{H}^n\big(G, E\big)$,  is equal to the $n$-th reduced cohomology w.r.t. to  the chain $L^p_{loc}(E)$, i.e.  $\overline{H}^{n,p}\big(G, E\big)$.

\textit{Step 3} : Using Step 2, we obtain that $\overline{H}^{n,p}(G, \tau_G^{(x)})=\overline{H}^n(G, \tau_G^{(x)})$ for all $x$ in $X$ and $\overline{H}^n\big(G, \int_X^{\oplus_p} \tau_G^{(x)} d\mu(x)\big)=\overline{H}^{n,p}\big(G, \int_X^{\oplus_p} \tau_G^{(x)} d\mu(x)\big)$. Hence,  we have our theorem using  Step 1 and Step 2.  \hfill\(\square\)

\end{document}